%% file: root.tex
\documentclass{article}

\usepackage{mathtools}
\usepackage{amssymb}
\usepackage{amsthm}
\usepackage{mytheorems}
\usepackage{mymacros}

\theoremstyle{plain}
\newtheorem*{TheoremA}{Theorem A}
\newtheorem*{TheoremB}{Theorem B}

\begin{document}

\title{\bf On the Glauberman-Solomon Theorem}

\author{Paul Flavell}

\maketitle

\input{intro}
\input{unique}
\input{csl}
\input{thmB}

\input{gss}
\input{bibliography}

\end{document}

%% file: intro.tex
\section{Introduction} \label{intro}

The purpose of this note is to give a complete and concise proof of the following:

\begin{TheoremA}[Glauberman]
    Let $p$ be an odd prime and $S \not= 1$ a $p$-group.
    Then $S$ possesses a characteristic subgroup $W(S) \not= 1$
    such that whenever $G$ is a $\qd{p}$-free group of
    characteristic $p$ then \[
        W(S) \normal G.
    \]
\end{TheoremA}

\noindent We remind the reader that:
\begin{itemize}
    \item   $\qd{p}$ is the semidirect product of $\spl{2}{p}$
            with its natural module.

    \item   $G$ is $\qd{p}$-free means it is not possible to find
            $K \normal H \leq G$ with $H/K \isom \qd{p}$.

    \item   $G$ has characteristic $p$ means $\cc{G}{\oo{p}{G}} \leq \oo{p}{G}$.
\end{itemize}

Glauberman's first proof of Theorem~A is in fact his $ZJ$-Theorem \cite{GG1}.
It is one of the foundational results in modern finite group theory.
There are numerous applications, particularly to the classification
of the finite simple groups.

The proof of the $ZJ$-Theorem involves
highly nontrivial commutator arguments. It is therefore natural to
seek alternative and simpler proofs.
Glauberman's $K^{\infty}$-Theorem \cite{GG2} gives a second proof,
followed by Puig's $L$-Theorem \cite{BG}.
Stellmacher gives two different approaches \cite{KS}, \cite{S1}.
Glauberman and Solomon give a remarkably simple proof in \cite{GS}.
A variation of their proof appears in \cite{FS} and it is that proof
that is presented here.

The key concept in all proofs of Theorem~A is that of {\em quadratic action}.
The action of an element $x$ on a $p$-group $P$ is said to be {\em quadratic}
if $[P,x,x] = 1$. If $P$ is elementary abelian then we may regard $P$ as
an $\gf{p}$-vectorspace, use additive notation and let $\BAR{x}$ denote the
endomorphism of $P$ corresponding to $x$.
Then $[P,x,x] = 1$ implies that $(\BAR{x}-1)^{2} = 0$.
Thus $\BAR{x}$ has quadratic minimal polynomial or $\BAR{x} = 1$.

The $\qd{p}$-free hypothesis is exploited by means of the following result,
which extends earlier work of Gorenstein-Walter \cite[section 4]{GW}.

\begin{TheoremB}[Glauberman \cite{GG1}]
    Let $p$ be an odd prime and $G$ a $\qd{p}$-free group.
    Suppose that $P \normal G$ is a $p$-group and $x \in G$ satisfies \[
        [P,x,x] = 1.
    \]
    Then $x$ acts trivially on every $G$-chief factor of $P$ and \[
        x \in \oo{p}{G \bmod \cc{G}{P}}.
    \]
\end{TheoremB}

The original proof of Theorem~B is also nontrivial.
It rests on a lengthy theorem of Dickson \cite{Gor} on generators for $\spl{2}{p^{n}}$.
Robinson \cite{R} gives a self-contained concise approach.
We present a variation of his argument.

%% file: unique.tex
\section{A uniqueness theorem} \label{unique}
The Baer-Suzuki Theorem,
which asserts that a conjugacy class generates a $p$-group if and only if
any pair of its elements generate a $p$-group,
is often used in the current context.
We require the following refinement, which is a special case of Wielandt's Zipper Lemma \cite{I1}.
The proof follows the Alperin-Lyons proof \cite{AL} of
the Baer-Suzuki Theorem.
All that is required is that if $Q$ is a non normal subgroup of the $p$-group
$P$ then $Q \not= Q^{g} \leq \nn{P}{Q}$ for some $g \in P$.
This follows by considering the conjugation action of $Q$ on the set of conjugates of $Q$.
Note that the Baer-Suzuki Theorem is a trivial corollary.

\begin{Theorem} \label{unique.1}
    Let $x$ be a $p$-element of the group $G$.
    Assume that \[
        \mbox{$x \not\in \oo{p}{G}$ and $x \in \oo{p}{H}$ whenever $x \in H < G$.}
    \]
    \begin{itemize}
        \item[(a)]    The element $x$ is contained in a unique maximal subgroup $M$ of $G$;

        \item[(b)]  $x^{G} \nsubseteq M$ and
                    $G = \listgen{x, x^{z}}$ for all $z \in x^{G} \setminus M$.
    \end{itemize}
\end{Theorem}
\begin{proof}
    Let $\Omega = x^{G}$ and define an $\Omega$-subgroup to be any proper
    subgroup of $G$ that is generated by elements of $\Omega$.
    Since every conjugate of $x$ obeys the same hypotheses as $x$,
    it follows that every $\Omega$-subgroup is a $p$-group.

    We claim that if $P$ and $Q$ are $\Omega$-subgroups with
    $Q \subsetneqq P$ then $\Omega \cap Q \subsetneqq \nn{\Omega \cap P}{Q}$.
    If $Q \normal P$ then this is clear so suppose $Q\notnormal P$.
    Since $P$ is a $p$-group it follows that
    $Q \not= Q^{g} \leq \nn{P}{Q}$ for some $g \in P$.
    Since $Q^{g}$ is generated by elements of $\Omega$,
    not all of these generators can be contained in $Q$
    and the claim holds.

    Assume (a) to be false.
    Then there exist distinct maximal subgroups $M_{1}$ and $M_{2}$
    that contain $x$.
    Choose such a pair with $\Omega \cap M_{1} \cap M_{2}$ maximal.
    For each $i$ set $P_{i} = \listgen{ \Omega \cap M_{i} }$.
    Then $P_{i}$ is a normal $p$-subgroup of $M_{i}$.
    Since $x \not\in \oo{p}{G}$ we have $M_{i} = \nn{G}{P_{i}}$.
    Let $Q = \listgen{ \Omega \cap M_{1} \cap M_{2} } \leq P_{1} \cap P_{2}$.
    As $x \in Q$ we have $\nn{G}{Q} \not= G$ and we may choose
    a maximal subgroup $N$ of $G$ with \[
        \nn{G}{Q} \leq N.
    \]

    Let $i \in \listset{1,2}$.
    If $Q = P_{i}$ then $M_{i} = \nn{G}{Q} \leq N$ and
    the maximality of $M_{i}$ forces $M_{i} = N$.
    Suppose that $Q \subsetneqq P_{i}$.
    By the claim \[
        \Omega \cap M_{1} \cap M_{2} = \Omega \cap Q %
        \subsetneqq \nn{\Omega \cap P_{i}}{Q} \subseteq \Omega \cap N \cap M_{i}.
    \]
    The choice of $M_{1},M_{2}$ forces $M_{i} = N$ in this case also.
    We deduce that $M_{1} = N = M_{2}$.
    This contradiction completes the proof of (a).

    Let $M$ be the unique maximal subgroup that contains $x$.
    If $x^{G} \subseteq M$ then $\listgen{x^{G}}$ is an $\Omega$-group
    and hence a $p$-group, contrary to $x \not\in \oo{p}{G}$.
    Then $x^{G} \nsubseteq M$.
    This implies that $M \not\normal G$ and $M = \nn{G}{M}$.
    Let $z \in x^{G}$ and suppose that $\listgen{x,x^{z}} \not= G$.
    Then $\listgen{x,x^{z}} \leq M$ so $x \in M^{z^{-1}}$ whence
    $M^{z^{-1}} = M$ and $z \in \nn{G}{M} = M$.
    This proves (b).
\end{proof}

%% file: csl.tex
\section{A characterization of \boldmath $\spl{2}{p}$} \label{csl}

\begin{Lemma} \label{csl.1}
    Let $V$ be a vectorspace and $x \in \EEnd{V}$ with $(x-1)^{2} = 0$.
    \begin{itemize}
        \item[(a)]  $V(x-1) \leq \cc{V}{x}$ and $\dim\cc{V}{x} \geq \frac{1}{2}\dim V$.

        \item[(b)]  Let $t \in \EEnd{V}$ have order $2$,
                    act as $-1$ on $\cc{V}{x}$ and $1$ on $V/\cc{V}{x}$.
                    Then $x^{t} = x^{-1}$.
    \end{itemize}
\end{Lemma}
\begin{proof}
    (a). We have $V(x - 1)(x - 1) = 0$ so $V(x - 1) \leq \cc{V}{x}$.
    Now $\dim V = \dim V(x - 1) + \dim\ker (x - 1) \leq 2\dim \cc{V}{x}$.

    (b). Let $s = tx$.
    Now $x$ acts as $1$ on $V/\cc{V}{x}$.
    Thus $s$ acts as $-1$ on $\cc{V}{x}$ and $1$ on $V/\cc{V}{x}$.
    Hence $V(s - 1) \leq \cc{V}{x}$ and $V(s - 1)(s + 1) = 0$.
    Then $s^{2} - 1 = 0$ and $txtx = 1$.
    Since $t^{-1} = t$ we have $x^{t} = x^{-1}$.
\end{proof}

The proof of the following is a variation on an argument of Robinson \cite{R}.

\begin{Theorem} \label{csl.2}
    Let $G$ be a group,
    $p$ an odd prime and
    $V$ a faithful irreducible $\gf{p}G$-module.
    Suppose that $G$ contains a $p$-element $x$ with the properties:
    \begin{itemize}
        \item  $x \in \oo{p}{H}$ whenever $x \in H < G$ and

        \item  $x$ acts quadratically on $V$.
    \end{itemize}
    Then $\dim V = 2$ and the image of $G$ in $\EEnd{V}$ is $\spll{V}$.
    In particular, $G \isom \spl{2}{p}$.
\end{Theorem}
\begin{proof}
    Identify $G$ with its image in $\EEnd{V}$.
    By irreducibility, $\oo{p}{G} = 1$ so
    Theorem~\ref{unique.1} implies that $x$ is contained in a unique maximal subgroup $M$
    of $G$ and there exists $z \in x^{G} \setminus M$ such that \[
        G = \listgen{x, x^{z}}.
    \]
    Now $\cc{V}{x} \cap \cc{V}{x^{z}} = \cc{V}{G} = 0$
    so $\dim\cc{V}{x} \leq \frac{1}{2}\dim V$.
    Lemma~\ref{csl.1} gives the opposite inequality,
    whence \[
        V = \cc{V}{x} \oplus \cc{V}{x^{z}}.
    \]
    Let $t \in \EEnd{V}$ act as $-1$ on $\cc{V}{x}$ and $1$ on $\cc{V}{x^{z}}$.
    Lemma~\ref{csl.1} implies that $x^{t} = x^{-1}$.
    It also implies that $-1_{V}t$ inverts $x^{z}$ and so $(x^{z})^{t} = (x^{z})^{-1}$ also.
    In particular $t$ normalizes $G$ and $z^{t} \in G$.
    Set \[
        H = \listgen{ z, z^{t} } \leq G.
    \]

    Choose $0 \not= e \in \cc{V}{x}$, let $f = ez \in \cc{V}{x^{z}}$
    and let $U = \listgen{e, f}$, so that $\dim U = 2$.
    Since $(z - 1)^{2} = 0$ we have $z^{2} = 2z - 1$.
    Thus \[
        ez = f, \quad fz = 2f - e,\quad  et = -e \qtext{and} ft = f.
    \]
    Hence $U$ is $\listgen{z,t}$-invariant and is an $\gf{p}H$-module.
    Also $\spll{U} \isom \spl{2}{p}$.

    Now $\cc{U}{z} = \listgen{ f - e }$ and $\cc{U}{z^{t}} = \listgen{ f + e }$
    so as $p \not= 2$, the images of $\listgen{z}$ and $\listgen{z^{t}}$
    in $\spll{U}$ are distinct Sylow $p$-subgroups.
    These generate $\spll{U}$,
    so the image of $H$ in $\spll{U}$ is $\spll{U}$.
    In particular, $H$ is not a $p$-group.

    Suppose that $H \not= G$.
    Since $z \in x^{G}$ it follows that $z \in \oo{p}{H}$.
    Then $H = \oo{p}{H}\listgen{z^{t}}$ and $H$ is a $p$-group.
    This contradiction forces $H = G$.
    Irreducibility forces $U = V$ and the proof is complete.
\end{proof}

%% file: thmB.tex
\section{The proof of Theorem~B} \label{thmB}

\begin{proof}[Proof of Theorem~B]
    Assume false and consider a minimal counterexample.
    Recall that \[
        \oo{p}{G} = \bigcap \cc{G}{V/U}
    \]
    where $V/U$ ranges over the $G$-chief factors of $P$.
    Hence there exists such a chief factor $V/U$
    on which $x$ acts nontrivially.
    Since $[P,x,x] = 1$ we have $[V,x,x] = 1$.
    The minimality of $G$ forces $U = 1$.
    Set \[
        \BAR{G} = G/\cc{G}{V}.
    \]
    Then $V$ is a faithful irreducible $\gf{p}\BAR{G}$-module.
    Suppose that $\BAR{x} \in \BAR{H} < \BAR{G}$.
    Let $H$ be the inverse image of $\BAR{H}$,
    so that $V \leq \cc{G}{V} \leq H < G$.
    The minimality of $G$ forces $x \in \oo{p}{H \bmod \cc{H}{V}}$
    whence $\BAR{x} \in \oo{p}{\BAR{H}}$.

    Theorem~\ref{csl.2} implies that $\dim V = 2$ and
    the image of $\BAR{G}$ in $\EEnd{V}$ is $\spll{V}$.
    Let $N$ be the inverse image of $\zenter{\BAR{G}}$ in $G$,
    so that $N \normal G$ and $\BAR{N} \isom \cyclic{2}$.
    Let $Z \in \syl{2}{N}$.
    Using the Frattini Argument we have \[
        N = Z\cc{G}{V} \qtext{and} G = \nn{G}{Z}N = \nn{G}{Z}\cc{G}{V}.
    \]
    Then $\nn{G}{Z}$ acts irreducibly on $V$ and as $\BAR{\nn{G}{Z}} = \BAR{G}$
    there exists a $p$-element $x_{0} \in \nn{G}{Z}$ with $\BAR{x_{0}} = \BAR{x}$.
    The minimality of $G$ forces $G = \nn{G}{Z}V$.
    If $\nn{G}{Z} \cap V \not= 1$ then irreducibility forces $V \leq \nn{G}{Z}$
    whence, as $p \not=2$, we obtain $[V,Z] \leq V \cap Z = 1$,
    contrary to $\BAR{N} \not= 1$.
    Thus $\nn{G}{Z} \cap V = 1$.

    Let $K = \nn{G}{Z} \cap \cc{G}{V} \normal G$.
    The minimality of $G$ forces $K = 1$.
    Then $\nn{G}{Z}$ is isomorphic to its image in $\EEnd{V}$ which is $\spll{V}$.
    Hence $G \isom \qd{p}$.
    This contradiction completes the proof.
\end{proof}

%% file: gss.tex
\section{The Glauberman-Solomon subgroup} \label{gss}
Let $S$ be a $p$-group for some prime $p$ and let $\enormalSgps{S}$
denote the set of elementary abelian normal subgroups of $S$.
Define 
\begin{align*}
    \GSS{S} &= \set{ V \in \enormalSgps{S} }{ \forall x \in S,\; [V,x,x] = 1 \implies [V,x] = 1} \quad\text{and} \\
    \gss{S} &= \listgen{\GSS{S}}.
\end{align*}
The subgroup $\gss{S}$ is called the {\em Glauberman-Solomon subgroup of $S$}.
It is clearly a characteristic subgroup.
Moreover $\OmegaOne{\zenter{S}} \in \GSS{S}$ so $\gss{S} \not= 1$ if $S \not= 1$.

Suppose that $U,V \in \GSS{S}$.
Then $[V,U,U] \leq [U,U] = 1$ and so $[V,U] = 1$.
Hence $UV \in \enormalSgps{S}$ and then $UV \in \GSS{S}$.
It follows that $\gss{S} \in \GSS{S}$ and so
$\gss{S}$ is the unique maximal member of $\GSS{S}$.
In other words,
\begin{quote} \em  
    $\gss{S}$ is the unique largest elementary abelian normal
    subgroup of $S$ that admits no nontrivial quadratic action from $S$.
\end{quote}
A useful property of $\gss{S}$ is that if $\gss{S} \leq T \leq S$ then
$\gss{S} \in \GSS{T}$ and so $\gss{S} \leq \gss{T}$.
We can now give a proof of Theorem~A.

\begin{Theorem}[Glauberman-Solomon]
    Let $G$ be a $\qd{p}$-free group of characteristic $p$ for the odd prime $p$.
    Let $S \in \syl{p}{G}$.
    Then \[
        \gss{S} \normal G.
    \]
\end{Theorem}
\begin{proof}
    Now $G$ has characteristic $p$ so $\cc{G}{\oo{p}{G}}$ is a $p$-group and 
    it follows that $\oo{p}{G \bmod \cc{G}{\oo{p}{G}}} = \oo{p}{G}$.
    Since $\oo{p}{G} \leq S$ and $\gss{S}$ is an abelian normal subgroup of $S$
    we have $[\oo{p}{G},\gss{S},\gss{S}] \leq [\gss{S},\gss{S}] = 1$.
    Theorem~B implies that $\gss{S} \leq \oo{p}{G}$.
    
    Let $V = \listgen{\gss{S}^{G}} \normal G, N = \oo{p}{G \bmod \cc{G}{V}} \normal G$
    and $S_{0} = S \cap N \in \syl{p}{N}$.
    Using the Frattini Argument we have \[
        N = S_{0}\cc{G}{V} \qtext{and} G = \nn{G}{S_{0}}N = \nn{G}{S_{0}}\cc{G}{V}.
    \]
    Now $\gss{S} \leq \oo{p}{G} \leq S_{0}$ so 
    $\gss{S} \leq \gss{S_{0}} \normal \nn{G}{S_{0}}$.
    We have \[
        V = \listgen{\gss{S}^{G}} = \listgen{\gss{S}^{\nn{G}{S_{0}}\cc{G}{V}}} %
        = \listgen{\gss{S}^{\nn{G}{S_{0}}}} \leq \gss{S_{0}}.
    \]
    We claim that $\gss{S_{0}} \in \GSS{S}$.
    Since $S_{0} \normal S$ we have $\gss{S_{0}} \in \enormalSgps{S}$.
    Let $x \in S$ and suppose that $[\gss{S_{0}},x,x] = 1$.
    Then $[V,x,x] = 1$ and Theorem~B forces $x \in S_{0}$.
    Hence $[\gss{S_{0}},x] = 1$ so $\gss{S_{0}} \in \GSS{S}$.
    We have \[
        \gss{S} \leq V \leq \gss{S_{0}} \leq \gss{S}.
    \]
    Then $\gss{S} = V \normal G$,
    completing the proof.

\end{proof}

%% file: bibliography.tex
\bibliographystyle{amsalpha}